\documentclass[utf8,reqno]{amsart}
\RequirePackage[utf8]{inputenc}
\usepackage{amsmath,amsthm,amsfonts,amssymb,amscd,amsbsy,multirow,hyperref}
\usepackage[all]{xy}
\usepackage{graphicx,pgf,caption,subcaption}
\usepackage{enumerate}
\usepackage{dsfont}
\usepackage{stmaryrd}
\usepackage{ytableau}
\usepackage{fancyhdr}
\usepackage{CJK}
\usepackage{hyperref}

\allowdisplaybreaks

\hypersetup{
    pdfnewwindow=true,
    colorlinks=true,
    linkcolor=blue,
    citecolor=blue,
    urlcolor=black,
}

\newtheorem{theorem}{Theorem}[]

\newtheorem{corollary}[theorem]{Corollary}

\theoremstyle{definition}
\newtheorem{definition}[theorem]{Definition}

\theoremstyle{remark}

\newtheorem{example}[theorem]{Example}

\allowdisplaybreaks
\numberwithin{equation}{section}
\numberwithin{theorem}{section}

\begin{document}
\begin{CJK*}{GBK}{song}

\title{On mixed super quasi-Einstein manifolds with Ricci-Bourguignon solitons}

\author[J. H. Yan]{Junhao Yan}
\address{Junhao Yan \\ School of Mathematics and Physics \\ Guangxi Minzu University \\ Guangxi Nanning 530006, China}
\email{1440116496@qq.com}

\author[R. Bi]{Ran Bi}
\address{Ran Bi \\ School of Mathematics and Physics \\ Guangxi Minzu University \\ Guangxi Nanning 530006, China}
\email{596915342@qq.com}

\author[W. J. Lu]{Weijun Lu}
\address{Weijun Lu \\ School of Mathematics and Physics \\ Guangxi Minzu University \\ Guangxi Nanning 530006, China}
\email{weijunlu2008@126.com}

\maketitle
\vspace{-0.3cm}

\begin{abstract}
This paper delves into the study of mixed super quasi-Einstein manifolds of dimension $n$ (for short, ${\rm M^{n}_{SQE}}$), focusing on their geometric and physical attributes. Initially, we explore several properties of ${\rm M^{n}_{SQE}}$,  including conformal Ricci pseudosymmetry, Einstein's field equation, and the space-matter tensor.  Subsequently, we characterize mixed super quasi-Einstein manifolds that admit Ricci-Bourguignon solitons.  We establish that if the generator vector field is torse-forming, the manifold reduces to a pseudo generalized quasi-Einstein manifold, and we provide a detailed characterization of the eigenvalue problem associated with the symmetric tensor. To substantiate our findings,  we construct an example to illustrate the existence of mixed super quasi-Einstein spacetime.
\end{abstract}


\section{Introduction}\par	
A Riemannian  manifold $(\mathcal{M}^{n},g)$ with dimension $n$ ($\geq2$) is said to be an Einstein manifold if the following condition
\begin{equation}\label{1.1}
\begin{split}
{\rm Ric}=\displaystyle\frac{r}{n}g
\end{split}
\end{equation}
holds on $\mathcal{M}$, where ${\rm Ric}$ and $r$  denote Ricci tensor and scalar curvature, respectively. Here (\ref{1.1}) is called the
Einstein metric condition. Einstein manifolds were named after Albert Einstein because this condition is equivalent to saying that the metric is a solution of the vacuum Einstein field equation with a cosmological constant. Therefore, they play an important role in  general relativity.

Chaki and Maity\cite{ref4} introduced the notion of quasi-Einstein manifold which is a generalization of Einstein manifold. A non-flat Riemannian manifold  $(\mathcal{M}^{n},g)$ $(n\geq3)$ is said to be a quasi-Einstein manifold if its Ricci tensor ${\rm Ric}$ of type $(0,2)$ is not identically zero with satisfying the following condition
\begin{equation}
\begin{split}
{\rm Ric}({ X},{Y})=\Psi_{1}g({ X},{ Y})+\Psi_{2}\mathcal{A}({ X})\mathcal{A}({Y}),
\end{split}
\end{equation}
where $\Psi_{1}$, $\Psi_{2}$ are scalars and $\Psi_{2}\neq0$. The vector field $\xi_{1}$ is called the generator of the manifold and $\mathcal{A}$ is a nonzero 1-form such that
\begin{equation*}
\begin{split}
g({ X},\xi_{1})=\mathcal{A}({ X}), \quad g(\xi_{1},\xi_{1})=\mathcal{A}(\xi_{1})=1,
\end{split}
\end{equation*}
for all  $X$ $\in {{C}^{\infty }}(TM)$. Such an $n$-dimensional manifold is simply denoted by ${\rm E^{n}_{Q}}$.

In general relativity, quasi-Einstein manifold can serve as a model for the perfect fluid spacetime\cite{ref5}. Therefore, it is of some importance in general relativity. Several geometric researchers have contributed to the development of quasi-Einstein manifolds by introducing various generalizations such as generalized quasi-Einstein manifolds\cite{ref2,ref8}, nearly quasi-Einstein manifolds\cite{ref7}, pseudo quasi-Einstein manifolds\cite{ref26}, pseudo generalized quasi-Einstein manifolds\cite{ref27}, super quasi-Einstein manifolds\cite{ref3}, mixed quasi-Einstein manifolds\cite{ref18}, mixed super quasi-Einstein manifolds\cite{ref1}, extended quasi-Einstein manifolds\cite{ref14} and so on.

One of the generalizations is the mixed super quasi-Einstein manifold introduced by Bhattacharyya, Tarafdar, and Debnath\cite{ref1}.
\begin{definition}
 A non-flat Riemannian manifold $(\mathcal{M}^{n},g)$ $(n\geq3)$ is called a mixed super quasi-Einstein manifold if its Ricci tensor Ric of type $(0,2)$ is not identically zero and satisfies the following condition
\begin{equation}\label{1.3}
\begin{split}
{\rm Ric}(X,Y)=&\Psi_{1}g(X,Y)+\Psi_{2}\mathcal{A}(X)\mathcal{A}(Y)+\Psi_{3}\mathcal{B}(X)\mathcal{B}(Y)\\
&+\Psi_{4}\big(\mathcal{A}(X)\mathcal{B}(Y)+\mathcal{B}(X)\mathcal{A}(Y)\big)+\Psi_{5}\mathcal{D}(X,Y),
\end{split}
\end{equation}
where $\Psi_{1}$, $\Psi_{2}$, $\Psi_{3}$, $\Psi_{4}$, $\Psi_{5}$ are scalars and $\Psi_{2}\neq0$, $\Psi_{3}\neq0$, $\Psi_{4}\neq0$, $\Psi_{5}\neq0$. The vector fields $\xi_{1}$ and $\xi_{2}$ are called the generators of the manifold. The associated 1-forms $\mathcal{A}$ and $\mathcal{B}$ are nonzero 1-forms satisfying
\begin{equation*}
\begin{split}
g({ X},\xi_{1})=\mathcal{A}({ X}),\quad g({ X},\xi_{2})=\mathcal{B}({ X}),  \quad g(\xi_{1},\xi_{2})=\mathcal{A}(\xi_{2})=\mathcal{B}({\xi_{1}})=0,\\
     g(\xi_{2},\xi_{2})=\mathcal{B}({\xi_{2}})=1, \quad g(\xi_{1},\xi_{1})=\mathcal{A}(\xi_{1})=1.
\end{split}
\end{equation*}
Furthermore, $\mathcal{D}$ is a symmetric tensor of type $(0,2)$ with trace-free, satisfying $\mathcal{D}(X,\xi_{1})=0$.
\end{definition}
If $\Psi_{4}$=$0$, then the manifold can reduce to a pseudo generalized quasi-Einstein manifold.

A Riemannian manifold $(\mathcal{M}^{n},g)$ ($ n\geq3$) is said to be a manifold of mixed super quasi-constant curvature if the curvature tensor ${\rm R}$  of type $(0,4)$ follows the condition
\begin{equation*}
\begin{split}
{\rm R}(X,Y,Z,W)=&\Phi_{1}\big(g(X,W)g(Y,Z)-g(X,Z)g(Y,W)\big)\\
&+\Phi_{2}\big(g(X,W)\mathcal{A}(Y)\mathcal{A}(Z)-g(X,Z)\mathcal{A}(Y)\mathcal{A}(W)\\
&+g(Y,Z)\mathcal{A}(X)\mathcal{A}(W)-g(Y,W)\mathcal{A}(X)\mathcal{A}(Z)\big)\\
&+\Phi_{2}\big(g(X,W)\mathcal{B}(Y)\mathcal{A}(Z)-g(X,Z)\mathcal{B}(Y)\mathcal{A}(W)\\
&+g(Y,Z)\mathcal{B}(X)\mathcal{A}(W)-g(Y,W)\mathcal{A}(X)\mathcal{B}(Z)\big)\\
&+\Phi_{4}\big[g(X,W)\big(\mathcal{A}(Y)\mathcal{B}(Z)+\mathcal{A}(Z)\mathcal{B}(Y)\big)\\
&-g(X,Z)\big(\mathcal{A}(Y)\mathcal{B}(W)+\mathcal{A}(W)\mathcal{B}(Y)\big)\\
&+g(Y,Z)\big(\mathcal{A}(X)\mathcal{B}(W)+\mathcal{A}(W)\mathcal{B}(X)\big)\\
&-g(Y,W)\big(\mathcal{A}(X)\mathcal{B}(Z)+\mathcal{A}(Z)\mathcal{B}(X)\big)\big]\\
&+\Phi_{5}\big(g(X,W)\mathcal{D}(Y,Z)-g(X,Z)\mathcal{D}(Y,W)\\
&+g(Y,Z)\mathcal{D}(X,W)-g(Y,W)\mathcal{D}(X,Z)\big),
\end{split}
\end{equation*}
where $\Phi_{1}, \Phi_{2}, \Phi_{3}, \Phi_{4}, \Phi_{5}$ are scalars, $\mathcal{A}$, $\mathcal{B}$ are two nonzero 1-forms and $\mathcal{D}$ is a symmetric tensor of type $(0,2)$. Such a manifold is simply denoted by ${\rm M^{n}_{SQC}}$.

In 2016, Dey, Pal and Bhattacharyya\cite{ref10} studied the mixed super quasi-Einstein manifolds satisfying some curvature conditions, such as ${\rm C\cdot Ric}$=$0$ and ${\rm \overline{C}\cdot Ric}$=$0$, where ${\rm C}$ is the conformal curvature tensor and ${\rm \overline{C}}$ is  the concircular curvature tensor. They
got the following theorem.
\begin{theorem} \cite[Theorem 4]{ref10}.
Let $(\mathcal{M}^{n},g)$  be a mixed super quasi-Einstein manifold. If the condition $ {\overline{C}\cdot Ric}$=$0$, then the curvature tensor $R$ satisfies
\begin{equation*}
\begin{split}
{\rm R}(X,Y,\xi_{1},\xi_{2})=\displaystyle\frac{r}{n(n-1)}\big(\mathcal{A}(Y)\mathcal{B}(X)-\mathcal{A}(X)\mathcal{B}(Y)\big).
\end{split}
\end{equation*}
\end{theorem}
In 2020, Debnath and Basu\cite{ref9} discussed some Ricci conditions on the mixed super quasi-Einstein manifolds, such as Ricci semisymmetry and Ricci pseudosymmetry. Several geometrical properties of mixed super quasi-Einstein manifolds have been studied by Fu, Han and Zhao\cite{ref12} and Patra and Patra\cite{ref20}. In 2024, Vasiulla, Ali and U\"{n}al\cite{ref29} investigated the sufficient condition for a mixed super quasi-Einstein manifold to be a manifold of mixed super quasi-constant curvature and obtained the following theorem.
\begin{theorem}
\cite[Theorem 1]{ref29}.
In a Riemannian manifold $(\mathcal{M}^{n},g)$, every quasi-conformally flat mixed super quasi-Einstein manifold is a manifold of mixed super quasi-constant curvature.
\end{theorem}

Based on  the above work, we study  the geometric and physical properties of mixed super quasi-Einstein manifolds. To generalize the work of Dey, Pal and Bhattacharyya on curvature conditions, we introduce the notion of conformal Ricci pseudosymmetric manifold. Furthermore, by extending the work of Huang, Zhou and Lu\cite{ref15} on several geometric properties of mixed quasi-Einstein manifolds with soliton structures, we consider the case of the Ricci-Bourguignon solitons. Within this framework, we derive the following results.
\begin{theorem}
Let a  mixed super quasi-Einstein manifold be conformal Ricci pseudosymmetric.
\begin{itemize}
	\item If $m$ is not an eigenvalue of the tensor $\mathcal{D}$ corresponding to the eigenvector $\xi_{2}$, then
\begin{equation*}
\begin{split}
{\rm R}(X,Y,\xi_{1},\xi_{2})=\mathcal{E}(X)\mathcal{A}(Y)-\mathcal{A}(X)\mathcal{E}(Y).
\end{split}
\end{equation*}
   \item  If $m$ is an eigenvalue of the tensor $\mathcal{D}$ corresponding to the eigenvector $\xi_{2}$, then
\begin{equation*}
\begin{split}
{\rm R}(X,Y,\xi_{1},\xi_{2})=0.
\end{split}
\end{equation*}
Where $m=-\displaystyle\frac{{\rm R}(\xi_{2},\xi_{1},\xi_{1},\xi_{2})(n-2)}{\Psi_{5}}+\mathcal{D}(\xi_{2},\xi_{2}).$
\end{itemize}
\end{theorem}
\begin{theorem}
Let $(\mathcal{M}^{n},g)$  be a  mixed super quasi-Einstein manifold  with Ricci-Bourguignon soliton $(g,\xi_{1},\lambda,\rho)$. If the generator $\xi_{1}$  is torse-forming vector field, then the  manifold $\mathcal{M}$ is a pseudo generalized
quasi-Einstein manifold and $\mathcal{D}(\xi_{2},\xi_{2})$ is an eigenvalue of  $\mathcal{D}$ corresponding to the eigenvector $\xi_{2}$.
\end{theorem}
This paper is organized as follows.

In Sec. 2, we study ${\rm M^{n}_{SQE}}$  with the generator $\xi_{1}$  as concircular  vector field. After that, we consider generalized Ricci recurrent and conformal Ricci pseudosymmetric manifold. In Sec. 3,  we discuss about  ${\rm M^{4}_{SQE}}$ spacetime satisfying Einstein field equation and  space-matter tensor, respectively. In Sec. 4, we consider Ricci-Bourguignon solitons on ${\rm M^{n}_{SQE}}$. Lastly, we construct  a non-trivial example of  ${\rm M^{4}_{SQE}}$ spacetime.

\section{Some Geometric Properties }

Let \{$e_{i}: i=1,2,...,n$\} be an orthonormal basis of the tangent space at any point of the manifold. Then putting $X=Y=e_{i}$ in (\ref{1.3}) and taking summation over $i$, $1\leq i\leq n$, we get
\begin{equation}\label{2.1}
\begin{split}
r=n\Psi_{1}+\Psi_{2}+\Psi_{3}.
\end{split}
\end{equation}

Next, we will consider some geometric properties of mixed super quasi-Einstein manifolds.

\subsection{The generator as concircular vector field}\

The generator $\xi_{1}$  is said to be concircular vector field if it satisfies\cite{ref25}
\begin{equation}\label{2.2}
\begin{split}
\nabla_{X}\xi_{1}=\mu X.
\end{split}
\end{equation}

Let us consider a mixed super quasi-Einstein manifold with concircular vector field  $\xi_{1}$. Hence, we can get the following theorem.
\begin{theorem}
Let the generator $\xi_{1}$ of a mixed super quasi-Einstein manifold be concircular vector field and the associated scalars be constants. Then the associated 1-form $\mathcal{A}$ is closed if and only if the associated 1-form $\mathcal{B}$ is closed, provide $\Psi_{1}+\Psi_{2}\neq 0$
\end{theorem}

\begin{proof}
Since the generator $\xi_{1}$ is a concircular vector field, then
\begin{equation*}
\begin{split}
{\rm R}(X,Y)\xi_{1}=(X\mu)Y-(Y\mu)X.
\end{split}
\end{equation*}
Consequently, the Ricci curvature satisfies
\begin{equation}\label{2.3}
\begin{split}
{\rm Ric}(X,\xi_{1})=(1-n)(X\mu).
\end{split}
\end{equation}
Putting $Y=\xi_{1}$ in (\ref{1.3}) gives
\begin{equation}\label{2.4}
\begin{split}
{\rm Ric}(X,\xi_{1})=(\Psi_{1}+\Psi_{2})\mathcal{A}(X)+\Psi_{4}\mathcal{B}(X).
\end{split}
\end{equation}
Combining (\ref{2.3}) and (\ref{2.4}), one obtains
\begin{equation}\label{2.5}
\begin{split}
X\mu=\frac{\Psi_{1}+\Psi_{2}}{1-n}\mathcal{A}(X)+\frac{\Psi_{4}}{1-n}\mathcal{B}(X).
\end{split}
\end{equation}
Noting that  the associated scalars are constants, then
\begin{align*}
 \nabla_{X}\nabla_{Y}\mu=&\frac{\Psi_{1}+\Psi_{2}}{1-n}\big((\nabla_{X}\mathcal{A})(Y)+\mathcal{A}(\nabla_{X}Y)\big)\\
&+\frac{\Psi_{4}}{1-n}\big((\nabla_{X}\mathcal{B})(Y)+\mathcal{B}(\nabla_{X}Y)\big),\\
\nabla_{Y}\nabla_{X}\mu=&\frac{\Psi_{1}+\Psi_{2}}{1-n}\big((\nabla_{Y}\mathcal{A})(X)+\mathcal{A}(\nabla_{Y}X)\big)\\
&+\frac{\Psi_{4}}{1-n}\big((\nabla_{Y}\mathcal{B})(X)+\mathcal{B}(\nabla_{Y}X)\big),\\
\nabla_{[X,Y]}\mu=&\frac{\Psi_{1}+\Psi_{2}}{1-n}\mathcal{A}([X,Y])+\frac{\Psi_{4}}{1-n}\mathcal{B}([X,Y]).
\end{align*}
According to
\begin{equation*}
\begin{split}
\nabla_{X}\nabla_{Y}\mu-\nabla_{Y}\nabla_{X}\mu-\nabla_{[X,Y]}\mu=0.
\end{split}
\end{equation*}
Thus
\begin{equation*}
\begin{split}
\frac{\Psi_{1}+\Psi_{2}}{1-n}\big((\nabla_{X}\mathcal{A})(Y)-(\nabla_{Y}\mathcal{A})(X)\big)+\frac{\Psi_{4}}{1-n}\big((\nabla_{X}\mathcal{B})(Y)-(\nabla_{Y}\mathcal{B})(X)\big)=0.
\end{split}
\end{equation*}
i.e.
\begin{equation*}
\begin{split}
(\Psi_{1}+\Psi_{2})d\mathcal{A}(X,Y)+\Psi_{4}d\mathcal{B}(X,Y)=0.
\end{split}
\end{equation*}
Due to $\Psi_{4}\neq 0$,  we have
\begin{equation*}
\begin{split}
\frac{\Psi_{1}+\Psi_{2}}{\Psi_{4}}d\mathcal{A}(X,Y)=-d\mathcal{B}(X,Y).
\end{split}
\end{equation*}
If $\Psi_{1}+\Psi_{2} \neq 0$, it follows that $d\mathcal{A}(X,Y)=0\Leftrightarrow d\mathcal{B}(X,Y)=0$.
\end{proof}

\subsection{Generalized Ricci recurrent ${\rm M^{n}_{SQE}}$ }\

A non-flat Riemannian manifold  $(\mathcal{M}^{n},g)$ $(n\geq3)$ is called a generalized Ricci recurrent manifold if its Ricci tensor ${\rm Ric}$ of type $(0,2)$  satisfies the following condition\cite{ref6}
\begin{equation}\label{2.6}
\begin{split}
(\nabla_{Z}{\rm Ric})(X,Y)=\gamma(Z){\rm Ric}(X,Y)+\delta(Z)g(X,Y),
\end{split}
\end{equation}
where $\gamma$ and $\delta$ are nonzero 1-forms. If $\delta=0$, then the manifold can be reduced to a Ricci recurrent manifold\cite{ref22}.

In this subsection, we consider a  mixed super quasi-Einstein manifold is generalized Ricci recurrent. Therefore, we can obtain the following theorem
\begin{theorem}
If a mixed super quasi-Einstein manifold is generalized Ricci recurrent, then
\begin{itemize}
	\item Either $\nabla_{Z}\xi_{1} \perp \xi_{2}$ or $\xi_{1}$ is parallel vector field if and only if  $Z(\Psi_{1}+\Psi_{2})-\gamma(Z)(\Psi_{1}+\Psi_{2})-\delta(Z)=0$.
	\item Either $\nabla_{Z}\xi_{2} \perp \xi_{1}$ or $\xi_{2}$ is parallel vector field if and only if $Z\big(\Psi_{1}+\Psi_{3}+\Psi_{5}\mathcal{D}(\xi_{2},\xi_{2})\big)-\gamma(Z)\big(\Psi_{1}+\Psi_{3}+\Psi_{5}\mathcal{D}(\xi_{2},\xi_{2})\big)-\delta(Z)-2\Psi_{5}\mathcal{D}(\nabla_{Z}\xi_{2},\xi_{2})=0$.
\end{itemize}
\end{theorem}
\begin{proof}
Since
\begin{equation}\label{2.7}
\begin{split}
(\nabla_{Z}{\rm Ric})(X,Y)=Z{\rm Ric}(X,Y)-{\rm Ric}(\nabla_{Z}X,Y)-{\rm Ric}(X,\nabla_{Z}Y).
\end{split}
\end{equation}
Thus, in view of (\ref{2.6}) and  (\ref{2.7}), we have
\begin{equation}\label{2.8}
\begin{split}
\gamma(Z){\rm Ric}(X,Y)+\delta(Z)g(X,Y)=&Z{\rm Ric}(X,Y)-{\rm Ric}(\nabla_{Z}X,Y)\\
&-{\rm Ric}(X,\nabla_{Z}Y).
\end{split}
\end{equation}
Using (\ref{1.3}) in (\ref{2.8}), we obtain
\begin{equation}\label{2.9}
\begin{split}
\gamma(Z)&\big[\Psi_{1}g(X,Y)+\Psi_{2}\mathcal{A}(X)\mathcal{A}(Y)+\Psi_{3}\mathcal{B}(X)\mathcal{B}(Y)+\Psi_{4}\big(\mathcal{A}(X)\mathcal{B}(Y)\\
&+\mathcal{B}(X)\mathcal{A}(Y)\big)+\Psi_{5}\mathcal{D}(X,Y)\big]+\delta(Z)g(X,Y)\\
=&Z\big[\Psi_{1}g(X,Y)+\Psi_{2}\mathcal{A}(X)\mathcal{A}(Y)+\Psi_{3}\mathcal{B}(X)\mathcal{B}(Y)\\
&+\Psi_{4}\big(\mathcal{A}(X)\mathcal{B}(Y)+\mathcal{B}(X)\mathcal{A}(Y)\big)+\Psi_{5}\mathcal{D}(X,Y)\big]\\
&-\big[\Psi_{1}g(\nabla_{Z}X,Y)+\Psi_{2}\mathcal{A}(\nabla_{Z}X)\mathcal{A}(Y)+\Psi_{3}\mathcal{B}(\nabla_{Z}X)\mathcal{B}(Y)\\
&+\Psi_{4}\big(\mathcal{A}(\nabla_{Z}X)\mathcal{B}(Y)+\mathcal{B}(\nabla_{Z}X)\mathcal{A}(Y)\big)+\Psi_{5}\mathcal{D}(\nabla_{Z}X,Y)\big]\\
&-\big[\Psi_{1}g(X,\nabla_{Z}Y)+\Psi_{2}\mathcal{A}(X)\mathcal{A}(\nabla_{Z}Y)+\Psi_{3}\mathcal{B}(X)\mathcal{B}(\nabla_{Z}Y)\\
&+\Psi_{4}\big(\mathcal{A}(X)\mathcal{B}(\nabla_{Z}Y)+\mathcal{B}(X)\mathcal{A}(\nabla_{Z}Y)\big)+\Psi_{5}\mathcal{D}(X,\nabla_{Z}Y)\big].
\end{split}
\end{equation}
Taking $X=Y=\xi_{1}$ in (\ref{2.9}) gives
\begin{equation*}
\begin{split}
Z(\Psi_{1}+\Psi_{2})-\gamma(Z)(\Psi_{1}+\Psi_{2})-\delta(Z)=2\Psi_{4}\mathcal{B}(\nabla_{Z}\xi_{1}).
\end{split}
\end{equation*}
Thus, $\mathcal{B}(\nabla_{Z}\xi_{1})=0 \Leftrightarrow Z(\Psi_{1}+\Psi_{2})-\gamma(Z)(\Psi_{1}+\Psi_{2})-\delta(Z)=0$. This implies that either $\nabla_{Z}\xi_{1} \perp \xi_{2}$ or $\xi_{1}$ is parallel vector field.

Again, taking $X=Y=\xi_{2}$ in (\ref{2.9}) gives
\begin{equation*}
\begin{split}
Z\big(\Psi_{1}&+\Psi_{3}+\Psi_{5}\mathcal{D}(\xi_{2},\xi_{2})\big)-\gamma(Z)\big(\Psi_{1}+\Psi_{3}+\Psi_{5}\mathcal{D}(\xi_{2},\xi_{2})\big)-\delta(Z)\\
&-2\Psi_{5}\mathcal{D}(\nabla_{Z}\xi_{2},\xi_{2})=2\Psi_{4}\mathcal{A}(\nabla_{Z}\xi_{2}).
\end{split}
\end{equation*}
Thus, $\mathcal{A}(\nabla_{Z}\xi_{2})=0 \Leftrightarrow Z\big(\Psi_{1}+\Psi_{3}+\Psi_{5}\mathcal{D}(\xi_{2},\xi_{2})\big)-\gamma(Z)\big(\Psi_{1}+\Psi_{3}+\Psi_{5}\mathcal{D}(\xi_{2},\xi_{2})\big)-\delta(Z)-2\Psi_{5}\mathcal{D}(\nabla_{Z}\xi_{2},\xi_{2})=0$. This implies that either $\nabla_{Z}\xi_{2} \perp \xi_{1}$ or $\xi_{2}$ is parallel vector field.
\end{proof}

\subsection{Conformal Ricci pseudosymmetry}\

We begin by recalling the notion of conformal curvature tensor ${\rm C}$.
(see \cite{ref17})
\begin{definition}
The conformal curvature tensor  ${\rm C}$ of type $(1,3)$ of an $n$-dimensional Riemannian manifold $(\mathcal{M}^{n},g)$ is defined by
\begin{equation}\label{2.10}
\begin{split}
{\rm C}(X,Y)Z=&{\rm R}(X,Y)Z-\displaystyle\frac{1}{n-2}\big({\rm Ric}(Y,Z)X-{\rm Ric}(X,Z)Y+g(Y,Z)QX\\
&-g(X,Z)QY\big)+\displaystyle\frac{r}{(n-1)(n-2)}\big(g(Y,Z)X-g(X,Z)Y\big),
\end{split}
\end{equation}
where $Q$ is the Ricci operator defined by ${\rm Ric}(X,Y)= g(QX,Y)$.
\end{definition}
The conformal curvature tensor ${\rm C}$ satisfies the following symmetry property:
\begin{equation}\label{2.11}
\begin{split}
{\rm C}(X,Y,Z,W)=-{\rm C}(X,Y,W,Z),
\end{split}
\end{equation}
where  ${\rm C}(X,Y,Z,W)=g({\rm C}(X,Y)Z,W)$.

In \cite{ref28}, the notion of ${\rm J}$-Ricci pseudosymmetric manifold was introduced by Shaikh and Kundu. If ${\rm J}$=${\rm C}$, then we have the following definition .
\begin{definition}
A Riemannian manifold $(\mathcal{M}^{n},g)$  is called  conformal Ricci pseudosymmetric manifold if its conformal curvature tensor  C satisfies
\begin{equation}\label{2.12}
\begin{split}
({\rm C}(X,Y)\cdot{\rm Ric})(Z,W) = F_{{\rm Ric}} Q(g,{\rm Ric})(Z,W;X,Y)
\end{split}
\end{equation}
holds on $U_{{\rm Ric}} = \{x\in\mathcal{M} :{\rm Ric}\neq\frac{r}{n}g$ at $x\}$, where $F_{{\rm Ric}}$ is a certain function on  $U_{{\rm Ric}}$.
\end{definition}

Next, we give the  definition of the eigenvector of the tensor $\mathcal{D}$ corresponding to the eigenvalue.
\begin{definition}
A vector field  $U$ is said to be an eigenvector of the tensor $\mathcal{D}$ corresponding to the eigenvalue $b$ if and only if
\begin{equation*}
\begin{split}
\mathcal{D}(X,U)=bg(X,U).	
\end{split}
\end{equation*}
\end{definition}

When we consider that a  mixed super quasi-Einstein manifold is conformal Ricci pseudosymmetric, we can obtain the following theorem.
\begin{theorem}
Let a  mixed super quasi-Einstein manifold be conformal Ricci pseudosymmetric.
\begin{itemize}
	\item If $m$ is not an eigenvalue of the tensor $\mathcal{D}$ corresponding to the eigenvector $\xi_{2}$, then
\begin{equation*}
\begin{split}
{\rm R}(X,Y,\xi_{1},\xi_{2})=\mathcal{E}(X)\mathcal{A}(Y)-\mathcal{A}(X)\mathcal{E}(Y).
\end{split}
\end{equation*}
   \item  If $m$ is an eigenvalue of the tensor $\mathcal{D}$ corresponding to the eigenvector $\xi_{2}$, then
\begin{equation*}
\begin{split}
{\rm R}(X,Y,\xi_{1},\xi_{2})=0.
\end{split}
\end{equation*}
Where $m=-\displaystyle\frac{{\rm R}(\xi_{2},\xi_{1},\xi_{1},\xi_{2})(n-2)}{\Psi_{5}}+\mathcal{D}(\xi_{2},\xi_{2}).$
\end{itemize}
\end{theorem}
\begin{proof}
Equation \ref{2.12} can be expressed as
\begin{equation}\label{2.13}
\begin{split}
{\rm Ric}(&{\rm C}(X,Y)Z,W)+{\rm Ric}(Z,{\rm C}(X,Y)W)\\
=&F_{{\rm Ric}}\big(g(Y,Z){\rm Ric}(X,W)-g(X,Z){\rm Ric}(Y,W)\\
&+g(Y,W){\rm Ric}(X,Z)-g(X,W){\rm Ric}(Y,Z)\big).
\end{split}
\end{equation}	
Using (\ref{1.3}) and (\ref{2.11}) in (\ref{2.13}), we have
\begin{equation}\label{2.14}
\begin{split}
\Psi_{2}&\big(\mathcal{A}({\rm C}(X,Y)Z)\mathcal{A}(W)+\mathcal{A}(Z)\mathcal{A}({\rm C}(X,Y)W)\big)\\
&+\Psi_{3}\big(\mathcal{B}({\rm C}(X,Y)Z)\mathcal{B}(W)+\mathcal{B}(Z)\mathcal{B}({\rm C}(X,Y)W)\big)\\
&+\Psi_{4}\big(\mathcal{A}({\rm C}(X,Y)Z)\mathcal{B}(W)+\mathcal{B}({\rm C}(X,Y)Z)\mathcal{A}(W)\\
&+\mathcal{A}(Z)\mathcal{B}({\rm C}(X,Y)W)+\mathcal{B}(Z)\mathcal{A}({\rm C}(X,Y)W)\big)\\
&+\Psi_{5}\big(\mathcal{D}({\rm C}(X,Y)Z,W)+\mathcal{D}(Z,{\rm C}(X,Y)W)\\
=&F_{{\rm Ric}}\big(g(Y,Z){\rm Ric}(X,W)-g(X,Z){\rm Ric}(Y,W)\\
&+g(Y,W){\rm Ric}(X,Z)-g(X,W){\rm Ric}(Y,Z)\big).
\end{split}
\end{equation}
We substitute $Z=W=\xi_{1}$ into (\ref{2.14})  to acquire
\begin{equation}\label{2.15}
\begin{split}
\Psi_{4}{\rm C}(X,Y,\xi_{1},\xi_{2})=\Psi_{4}F_{{\rm Ric}}\big(\mathcal{B}(X)\mathcal{A}(Y)-\mathcal{A}(X)\mathcal{B}(Y)\big).
\end{split}
\end{equation}
Since $\Psi_{4} \neq 0$, thus (\ref{2.15}) turns into
\begin{equation}\label{2.16}
\begin{split}
{\rm R}(X,Y,\xi_{1},\xi_{2})=&\big(F_{{\rm Ric}}+\displaystyle\frac{2\Psi_{1}+\Psi_{2}+\Psi_{3}}{n-2}-\displaystyle\frac{r}{(n-1)(n-2)}\big)\big(\mathcal{B}(X)\mathcal{A}(Y)\\
&-\mathcal{A}(X)\mathcal{B}(Y)\big)+\displaystyle\frac{\Psi_{5}}{n-2}\big(\mathcal{D}(X,\xi_{2})\mathcal{A}(Y)\\
&-\mathcal{A}(X)\mathcal{D}(Y,\xi_{2})\big).
\end{split}
\end{equation}
Putting $X=\xi_{2}$ and $Y=\xi_{1}$ into (\ref{2.16}), then
\begin{equation}\label{2.17}
\begin{split}
F_{{\rm Ric}}=&{\rm R}(\xi_{2},\xi_{1},\xi_{1},\xi_{2})+\displaystyle\frac{r}{(n-1)(n-2)}-\displaystyle\frac{2\Psi_{1}+\Psi_{2}+\Psi_{3}}{n-2}\\
&-\displaystyle\frac{\Psi_{5}}{n-2}\mathcal{D}(\xi_{2},\xi_{2}).
\end{split}
\end{equation}
Combining (\ref{2.16}) and (\ref{2.17}), we can acquire
\begin{equation*}
\begin{split}
{\rm R}(X,Y,\xi_{1},\xi_{2})=&\big(({\rm R}(\xi_{2},\xi_{1},\xi_{1},\xi_{2})-\displaystyle\frac{\Psi_{5}}{n-2}\mathcal{D}(\xi_{2},\xi_{2}))\mathcal{B}(X)+\displaystyle\frac{\Psi_{5}}{n-2}\mathcal{D}(X,\xi_{2})\big)\mathcal{A}(Y)\\
&-\big(({\rm R}(\xi_{2},\xi_{1},\xi_{1},\xi_{2})-\displaystyle\frac{\Psi_{5}}{n-2}\mathcal{D}(\xi_{2},\xi_{2}))\mathcal{B}(Y)\\
&+\displaystyle\frac{\Psi_{5}}{n-2}\mathcal{D}(Y,\xi_{2})\big)\mathcal{A}(X),
\end{split}
\end{equation*}
i.e.
\begin{equation}\label{2.18}
\begin{split}
{\rm R}(X,Y,\xi_{1},\xi_{2})=\mathcal{E}(X)\mathcal{A}(Y)-\mathcal{E}(Y)\mathcal{A}(X),
\end{split}
\end{equation}
where $\mathcal{E}(X)=g(X,\vartheta)=({\rm R}(\xi_{2},\xi_{1},\xi_{1},\xi_{2})-\displaystyle\frac{\Psi_{5}}{n-2}\mathcal{D}(\xi_{2},\xi_{2}))\mathcal{B}(X)+\displaystyle\frac{\Psi_{5}}{n-2}\mathcal{D}(X,\xi_{2})$ for all $X$ $\in {{C}^{\infty }}(TM)$, provide $\mathcal{E}\neq 0$.

If $\mathcal{E}(X)=0$, we can obtain
\begin{equation*}
\begin{split}
\mathcal{D}(X,\xi_{2})=\big(-\displaystyle\frac{{\rm R}(\xi_{2},\xi_{1},\xi_{1},\xi_{2})(n-2)}{\Psi_{5}}+\mathcal{D}(\xi_{2},\xi_{2})\big)g(X,\xi_{2}),
\end{split}
\end{equation*}
which implies that $\big(-\frac{{\rm R}(\xi_{2},\xi_{1},\xi_{1},\xi_{2})(n-2)}{\Psi_{5}}+\mathcal{D}(\xi_{2},\xi_{2})\big)$ is an eigenvalue of the tensor $\mathcal{D}$ corresponding to the eigenvector $\xi_{2}$.

From (\ref{2.18}), it follows that
\begin{equation*}
\begin{split}
{\rm R}(X,Y,\xi_{1},\xi_{2})=0.
\end{split}
\end{equation*}
\end{proof}

\section{Some Physical Properties}

The Einstein's field equation  without a cosmological constant is denoted by\cite{ref19}
\begin{equation}\label{3.1}
\begin{split}
{\rm Ric}(X,Y)-\displaystyle\frac{r}{2}g(X,Y)=\kappa{\rm T}(X,Y),
\end{split}
\end{equation}
where $\kappa$ is a gravitational constant and ${\rm T}$ is the energy-momentum tensor.

Petrov\cite{ref23} introduced the space-matter tensor ${\rm P}$ of type $(0,4)$  which can be expressed as
\begin{equation}\label{3.2}
\begin{split}
{\rm P}={\rm R}+\frac{\kappa}{2}g\wedge {\rm T}-\sigma{\rm G},
\end{split}
\end{equation}
where ${\rm R}$ is the curvature tensor of type (0,4), $\sigma$ is the energy density as well as  $g\wedge {\rm T}$  is given by
\begin{equation}\label{3.3}
\begin{split}
g\wedge {\rm T}(X,Y,Z,W)=&g(X,W){\rm T}(Y,Z)+g(Y,Z){\rm T}(X,W)\\
&-g(X,Z){\rm T}(Y,W)-g(Y,W){\rm T}(X,Z),
\end{split}
\end{equation}
and ${\rm G}$ is defined by
\begin{equation}\label{3.4}
\begin{split}
{\rm G}(X,Y,Z,W)= g(X,W)g(Y,Z)-g(X,Z)g(Y,W),
\end{split}
\end{equation}
for all  $X,Y,Z,W$ $\in {{C}^{\infty }}(TM)$.

In this section, we will give some physical properties of the mixed super quasi-Einstein spacetime.

\subsection{Einstein's field equation in the $M^{4}_{SQE}$ spacetime }\

Let us consider that the associated 1-forms $\mathcal{A}$ and $\mathcal{B}$, the symmetric tensor $\mathcal{D}$ and the energy-momentum tensor ${\rm T}$ are covariant constant. Then, we have
\begin{align}
(\nabla_{Z}\mathcal{A})(X)=0,\quad &(\nabla_{Z}\mathcal{B})(X)=0,\quad (\nabla_{Z}\mathcal{D})(X,Y)=0,\label{3.5}\\
&(\nabla_{Z}{\rm T})(X,Y)=0.\label{3.6}
\end{align}

In the subsection, we study the  mixed super quasi-Einstein spacetime that admits Einstein's field equation without a cosmological constant, and subsequently we obtain the following theorem.
\begin{theorem}
Let a connected pseudo-Riemannian manifold $(\mathcal{M}^{4},g)$ be the mixed super quasi-Einstein spacetime satisfying Einstein's field equation without a cosmological constant. If the 1-forms $\mathcal{A}, \mathcal{B}$ and the tensor $\mathcal{D}$ are covariant constant. Suppose that the associated scalars are constants. Then the energy-momentum tensor $T$ is also covariant constant.
\end{theorem}
\begin{proof}
By virtue of (\ref{1.3}) and (\ref{3.1}), we can find
\begin{equation}\label{3.7}
\begin{split}
(\Psi_{1}-\displaystyle\frac{ r}{2})g(X,Y)&+\Psi_{2}\mathcal{A}(X)\mathcal{A}(Y)+\Psi_{3}\mathcal{B}(X)\mathcal{B}(Y)+\Psi_{4}\big(\mathcal{A}(X)\mathcal{B}(Y)\\
&+\mathcal{B}(X)\mathcal{A}(Y)\big)+\Psi_{5}\mathcal{D}(X,Y)=\kappa{\rm T}(X,Y).
\end{split}
\end{equation}
Now, taking the covariant derivative of (\ref{3.7}) with respect to $Z$, we arrive at
\begin{equation}\label{3.8}
\begin{split}
Z(&\Psi_{1}-\displaystyle\frac{ r}{2})g(X,Y)+Z(\Psi_{2})\mathcal{A}(X)\mathcal{A}(Y)+\Psi_{2}\big((\nabla_{Z}\mathcal{A})(X)\mathcal{A}(Y)\\
&+(\nabla_{Z}\mathcal{A})(Y)\mathcal{A}(X)\big)+Z(\Psi_{3})\mathcal{B}(X)\mathcal{B}(Y)+\Psi_{3}\big((\nabla_{Z}\mathcal{B})(X)\mathcal{B}(Y)\\
&+(\nabla_{Z}\mathcal{B})(Y)\mathcal{B}(X)\big)+Z(\Psi_{4})\big(\mathcal{A}(X)\mathcal{B}(Y)+\mathcal{B}(X)\mathcal{A}(Y)\big)\\
&+\Psi_{4}\big((\nabla_{Z}\mathcal{A})(X)\mathcal{B}(Y)+(\nabla_{Z}\mathcal{B})(Y)\mathcal{A}(X)+(\nabla_{Z}\mathcal{A})(Y)\mathcal{B}(X)\\
&+(\nabla_{Z}\mathcal{B})(X)\mathcal{A}(Y)\big)+Z(\Psi_{5})\mathcal{D}(X,Y)+\Psi_{5}(\nabla_{Z}\mathcal{D})(X,Y)\\
=&\kappa(\nabla_{Z}{\rm T})(X,Y).
\end{split}
\end{equation}
Since the associated scalars are constants, (\ref{3.8}) turns into
\begin{equation}\label{3.9}
\begin{split}
\Psi_{2}&\big((\nabla_{Z}\mathcal{A})(X)\mathcal{A}(Y)+(\nabla_{Z}\mathcal{A})(Y)\mathcal{A}(X)\big)+\Psi_{3}\big((\nabla_{Z}\mathcal{B})(X)\mathcal{B}(Y)\\
&+(\nabla_{Z}\mathcal{B})(Y)\mathcal{B}(X)\big)+\Psi_{4}\big((\nabla_{Z}\mathcal{A})(X)\mathcal{B}(Y)+(\nabla_{Z}\mathcal{B})(Y)\mathcal{A}(X)\\
&+(\nabla_{Z}\mathcal{A})(Y)\mathcal{B}(X)+(\nabla_{Z}\mathcal{B})(X)\mathcal{A}(Y)\big)+\Psi_{5}(\nabla_{Z}\mathcal{D})(X,Y)\\
=&\kappa(\nabla_{Z}{\rm T})(X,Y).
\end{split}
\end{equation}
Using (\ref{3.5}) into (\ref{3.9}), we can get (\ref{3.6}). Our theorem is thus proved.
\end{proof}

\subsection{The vanishing space-matter tensor}\

In the subsection, we consider the mixed super quasi-Einstein spacetime with a vanishing space-matter tensor. In \cite{ref29},  Vasiulla, Ali and \"{U}nal gave the following theorem.
\begin{theorem}\cite[Theorem 10]{ref29}.
The mixed super quasi-Einstein spacetime with a vanishing space-matter tensor  satisfying Einstein's field equation without a cosmological constant is a spacetime of mixed super quasi-constant curvature.
\end{theorem}
\begin{proof}
With the help of (\ref{3.2})-(\ref{3.4}), we can obtain
\begin{equation}\label{3.10}
\begin{split}
{\rm P}(X,Y,Z,W)=&{\rm R}(X,Y,Z,W)+\displaystyle\frac{\kappa}{2}\big(g(X,W){\rm T}(Y,Z)\\
&+g(Y,Z){\rm T}(X,W)\big)-g(X,Z){\rm T}(Y,W)\\
&-g(Y,W){\rm T}(X,Z)-\sigma\big(g(X,W)g(Y,Z)\\
&-g(X,Z)g(Y,W)\big).
\end{split}
\end{equation}
Making use of  ${\rm P}=0$, then (\ref{3.10}) takes the form
\begin{equation}\label{3.11}
\begin{split}
{\rm R}(X,Y,Z,W)=&-\displaystyle\frac{\kappa}{2}\big(g(X,W){\rm T}(Y,Z)+g(Y,Z){\rm T}(X,W)\big)\\
&-g(X,Z){\rm T}(Y,W)-g(Y,W){\rm T}(X,Z)\\
&+\sigma\big(g(X,W)g(Y,Z)-g(X,Z)g(Y,W)\big).
\end{split}
\end{equation}
From (\ref{3.7}), it follows that
\begin{equation}\label{3.12}
\begin{split}
{\rm T}(X,Y)=&\big(\displaystyle\frac{\Psi_{1}-\frac{r}{2}}{\kappa}\big)g(X,Y)+
\displaystyle\frac{\Psi_{2}}{\kappa}\mathcal{A}(X)\mathcal{A}(X)+\displaystyle\frac{\Psi_{3}}{\kappa}\mathcal{B}(X)\mathcal{B}(Y)\\
&+\displaystyle\frac{\Psi_{4}}{\kappa}\big(\mathcal{A}(X)\mathcal{B}(Y)+\mathcal{B}(X)\mathcal{A}(Y)\big)+\displaystyle\frac{\Psi_{5}}{\kappa}\mathcal{D}(X,Y).
\end{split}
\end{equation}
Here, by combining (\ref{3.11}) and (\ref{3.12}) to obtain
\begin{equation}\label{3.13}
\begin{split}
{\rm R}(X,Y,Z,W)=&f_{1}\big(g(X,W)g(Y,Z)-g(X,Z)g(Y,W)\big)\\
&+f_{2}\big(g(X,W)\mathcal{A}(Y)\mathcal{A}(Z)-g(X,Z)\mathcal{A}(Y)\mathcal{A}(W)\\
&+g(Y,Z)\mathcal{A}(X)\mathcal{A}(W)-g(Y,W)\mathcal{A}(X)\mathcal{A}(Z)\big)\\
&+f_{3}\big(g(X,W)\mathcal{B}(Y)\mathcal{B}(Z)-g(X,Z)\mathcal{B}(Y)\mathcal{B}(W)\\
&+g(Y,Z)\mathcal{B}(X)\mathcal{B}(W)-g(Y,W)\mathcal{B}(X)\mathcal{B}(Z)\big)\\
&+f_{4}\big[g(X,W)\big(\mathcal{A}(Y)\mathcal{B}(Z)+\mathcal{A}(Z)\mathcal{B}(Y)\big)\\
&-g(X,Z)\big(\mathcal{A}(Y)\mathcal{B}(W)+\mathcal{A}(W)\mathcal{B}(Y)\big)\\
&+g(Y,Z)\big(\mathcal{A}(X)\mathcal{B}(W)+\mathcal{A}(W)\mathcal{B}(X)\big)\\
&-g(Y,W)\big(\mathcal{A}(X)\mathcal{B}(Z)+\mathcal{A}(Z)\mathcal{B}(X)\big)\big]\\
&+f_{5}\big(g(X,W)\mathcal{D}(Y,Z)-g(X,Z)\mathcal{D}(Y,W)\\
&+g(Y,Z)\mathcal{D}(X,W)-g(Y,W)\mathcal{D}(X,Z)\big),
\end{split}
\end{equation}
where
\begin{equation*}
\begin{split}
f_{1}&=\left(\displaystyle\frac{r}{2}-\Psi_{1}+\sigma\right), \quad f_{2}=-\displaystyle\frac{\Psi_{2}}{2},\quad f_{3}=-\displaystyle\frac{\Psi_{3}}{2},\\
f_{4}&=-\displaystyle\frac{\Psi_{4}}{2},\quad f_{5}=-\displaystyle\frac{\Psi_{5}}{2},
\end{split}
\end{equation*}
which means that the spacetime is a spacetime of mixed super quasi-constant curvature.
\end{proof}
Next, let $(\xi_{1}, \xi_{2})^\perp$ denote the $(n-2)$-dimensional distribution in the mixed super quasi-Einstein spacetime. Then $g(X,\xi_{1})=0, g(X,\xi_{2})=0$ if $X \in (\xi_{1}, \xi_{2})^\perp$. Hence, we can get following theorem.
\begin{theorem}
Let a connected pseudo-Riemannian manifold $(\mathcal{M}^{4},g)$ be the mixed super quasi-Einstein spacetime with a vanishing space-matter tensor  satisfying Einstein's field equation without a cosmological constant. If $\mathcal{M}^{4}$ is homogenous with respect to the symmetric tensor  $\mathcal{D}$ in the direction of $X$, then sectional curvature of the plane determined by two vectors $X \in (\xi_{1}, \xi_{2})^\perp$ and $\xi_{1}$ is
\begin{equation*}
\begin{split}
\Psi_{1}+\frac{\Psi_{3}}{2}+\sigma+\beta.
\end{split}
\end{equation*}
\end{theorem}

\begin{proof}
Since $X \in (\xi_{1}, \xi_{2})^\perp$. By using (\ref{3.13}), we can write
\begin{equation}\label{3.14}
\begin{split}
{\rm R}(X,\xi_{1},\xi_{1},X)=&\left(\displaystyle\frac{r}{2}-\Psi_{1}+\sigma\right)\big(g(X,X)g(\xi_{1},\xi_{1})\big)-\displaystyle\frac{\Psi_{2}}{2}g(X,X)\\
&+g(\xi_{1},\xi_{1})\mathcal{D}(X,X).
\end{split}
\end{equation}
Noting that $\mathcal{M}^{4}$ is homogenous with respect to the symmetric tensor $\mathcal{D}$ in the direction of $X$, then
\begin{equation}\label{3.15}
\begin{split}
\mathcal{D}(X,X)=\beta g(X,X),
\end{split}
\end{equation}
where $\beta$ is a scalar.

In the mixed super quasi-Einstein spacetime, we have
\begin{equation}\label{3.16}
\begin{split}
r=4\Psi_{1}-\Psi_{2}+\Psi_{3},\quad g(\xi_{1},\xi_{1})=-1.
\end{split}
\end{equation}
Substituting {\ref{3.15}} and {\ref{3.16}} into {\ref{3.14}},  we obtain the sectional curvature of the plane determined by two vector $X \in (\xi_{1}, \xi_{2})^\perp$ and $\xi_{1}$ is
\begin{equation*}
\begin{split}
K(X,\xi_{1})=-\displaystyle\frac{{\rm R}(X,\xi_{1},\xi_{1},X)}{g(X,X)}=\Psi_{1}+\frac{\Psi_{3}}{2}+\sigma+\beta.
\end{split}
\end{equation*}
\end{proof}

\subsection{The divergence-free space-matter tensor}\  

In this subsection, we look for sufficient conditions under which the divergence of the space-matter tensor vanishes.
\begin{theorem}
Let the associated scalars and  the energy density in the mixed super quasi-Einstein spacetime satisfying Einstein's field equation without a  cosmological constant be constants. Suppose that  the tensor $\mathcal{D}$ is of  Codazzi type. If the generators $\xi_{1}$ and $\xi_{2}$ are parallel vector fields. Then the space-matter tensor will be divergence-free.
\end{theorem}

\begin{proof}
From (\ref{3.1}) and (\ref{3.10}), we obtain
\begin{equation}\label{3.17}
\begin{split}
({\rm div}{\rm P})(X,Y,Z)=&({\rm div}{\rm R})(X,Y,Z)+\displaystyle\frac{1}{2}\big((\nabla_{X}{\rm Ric})(Y,Z)\\
&-(\nabla_{Y}{\rm Ric})(X,Z)\big)-g(Y,Z)\big(\frac{1}{4}X(r)+X(\sigma)\big)\\
&+g(X,Z)\big(\displaystyle\frac{1}{4}Y(r)+Y(\sigma)\big).
\end{split}
\end{equation}
For a pseudo-Riemannian manifold, it is known that
\begin{equation}\label{3.18}
\begin{split}
({\rm div}{\rm R})(X,Y,Z)=\big((\nabla_{X}{\rm Ric})(Y,Z)-(\nabla_{Y}{\rm Ric})(X,Z)\big)
\end{split}
\end{equation}
Combining (\ref{3.17}) and (\ref{3.18}), we have
\begin{equation}\label{3.19}
\begin{split}
({\rm div}{\rm R})(X,Y,Z)=&\displaystyle\frac{3}{2}\big((\nabla_{X}{\rm Ric})(Y,Z)-(\nabla_{Y}{\rm Ric})(X,Z)\big)\\
&-g(Y,Z)\big(\frac{1}{4}X(r)+X(\sigma)\big)\\
&+g(X,Z)\big(\displaystyle\frac{1}{4}Y(r)+Y(\sigma)\big).
\end{split}
\end{equation}
Since the tensor $\mathcal{D}$ is of  Codazzi type, then
\begin{equation}\label{3.20}
\begin{split}
(\nabla_{X}\mathcal{D})(Y,Z)=(\nabla_{Y}\mathcal{D})(X,Z).
\end{split}
\end{equation}
In view of (\ref{1.3}) and (\ref{3.20}), (\ref{3.19}) becomes
\begin{equation}\label{3.21}
\begin{split}
({\rm div}{\rm P})(X,Y,Z)=&\frac{3}{2}\big[\Psi_{2}\big((\nabla_{X}\mathcal{A})(Y)\mathcal{A}(Z)+(\nabla_{\rm X}\mathcal{A})(Z)\mathcal{A}(Y)\\
&-(\nabla_{Y}\mathcal{A})(X)\mathcal{A}(Z)-(\nabla_{Y}\mathcal{A})(Z)\mathcal{A}(X)\big)\\
&+\Psi_{3}\big((\nabla_{X}\mathcal{B})(Y)\mathcal{B}(Z)+(\nabla_{X}\mathcal{B})(Z)\mathcal{B}(Y)\\
&-(\nabla_{Y}\mathcal{B})(X)\mathcal{B}(Z)-(\nabla_{Y}\mathcal{B})(Z)\mathcal{B}(X)\big)\\
&+\Psi_{4}\big((\nabla_{X}\mathcal{A})(Y)\mathcal{B}(Z)+(\nabla_{X}\mathcal{B})(Y)\mathcal{A}(Z)\\
&+(\nabla_{X}\mathcal{A})(Z)\mathcal{B}(Y)+(\nabla_{X}\mathcal{B})(Z)\mathcal{A}(Y)\\
&-(\nabla_{Y}\mathcal{A})(X)\mathcal{B}(Z)-(\nabla_{Y}\mathcal{B})(X)\mathcal{A}(Z)\\
&-(\nabla_{Y}\mathcal{A})(Z)\mathcal{B}(X)-(\nabla_{Y}\mathcal{B})(Z)\mathcal{A}(X)\big)\big]\\
&-g(Y,Z)\big(\frac{1}{4}X(r)+X(\sigma)\big)\\
&+g(X,Z)\big(\displaystyle\frac{1}{4}Y(r)+Y(\sigma)\big).
\end{split}
\end{equation}
Noticing that the associated scalars and the energy density $\sigma$ are constants and the generators  $\xi_{1}$ and $\xi_{2}$ are parallel vector fields, then
\begin{equation*}
\begin{split}
Y(r)=0,\quad Y(\sigma)=0,\quad (\nabla_{X}\mathcal{A})(Y)=0, \quad  (\nabla_{X}\mathcal{B})(Y)=0.
\end{split}
\end{equation*}
Therefore, from (\ref{3.21}), it follows that
\begin{equation*}
\begin{split}
({\rm div}{\rm P})(X,Y,Z)=0.
\end{split}
\end{equation*}
\end{proof}
Supposing that $({\rm div}{\rm P})(X,Y,Z)=0$ and then contracting (\ref{3.14}) over $Y$ and $Z$, we can obtain
\begin{equation*}
\begin{split}
X(\sigma)=0,
\end{split}
\end{equation*}
we conclude that the energy density is a constant. Hence, we have the following corollary.
\begin{corollary}
For a divergence-free space-matter tensor, the energy density in the mixed super quasi-Einstein spacetime obeying Einstein's field equation without a cosmological constant is a constant.
\end{corollary}

\section{Ricci-Bourguignon Solitons on ${\rm M^{n}_{SQE}}$ }

In the section, we first recall the definition of Ricci-Bourguignon solitons (see \cite[Definition 1.1]{ref11}).
\begin{definition}
A quadruplet $(g,U,\lambda,\rho)$ on non-flat Riemannian manifold $(\mathcal{M}^{n},g)$  is a Ricci-Bourguignon soliton if there exist a vector field $U$ and a  constant $\lambda$  such that
\begin{equation}\label{4.1}
\begin{split}
\displaystyle\frac{1}{2}\mathcal{L}_{U}g+{\rm Ric}=(\lambda+\rho r)g,
\end{split}
\end{equation}
where $\mathcal{L}_{U}$ denotes the Lie derivative operator in the direction of the vector field $U$.
\end{definition}

In \cite{ref11}, for different values of $\rho$, the Ricci-Bourguignon soliton can reduce to different types of solitons. Such as
\begin{itemize}
	\item \( \rho=0 \): the soliton is Ricci soliton\cite{ref13};\\
	
	\item \( \rho=\displaystyle\frac{1}{2} \): the soliton is  Einstein soliton;\\
	
	\item \( \rho=\displaystyle\frac{1}{n} \): the soliton is traceless Ricci soliton;\\
	
	\item \( \rho=\displaystyle\frac{1}{2(n-1)} \): the soliton is Schouten soliton\cite{ref24}.
\end{itemize}

A Ricci-Bourguignon soliton is called expanding, steady or shrinking depending on whether $\lambda>0,=0,<0$. Some characteristics of Ricci-Bourguignon solitons have been studied by  Patra\cite{ref21}.

A vector field  $U$ is called  a torse-forming vector field if it satisfies\cite{ref30}
\begin{equation}\label{4.2}
\begin{split}
\nabla_{X}U=fX+\alpha(X)U,
\end{split}
\end{equation}
where $f$ is a function, $\alpha$ is a 1-form.

The conharmonic  curvature tensor  ${\rm \tilde{C}}$ of type $(1,3)$ of an $n$-dimensional Riemannian manifold $(\mathcal{M}^{n},g)$ is defined by\cite{ref16}
\begin{equation*}
\begin{split}
{\rm \tilde{C}}(X,Y)Z=&{\rm R}(X,Y)Z-\displaystyle\frac{1}{n-2}\big({\rm Ric}(Y,Z)X-{\rm Ric}(X,Z)Y+g(Y,Z)QX\\
&-g(X,Z)QY\big),
\end{split}
\end{equation*}
where $Q$ is the Ricci operator defined by ${\rm Ric}(X,Y)=g(QX,Y)$. If ${\rm \tilde{C}}=0$, then the manifold is conharmonic flat.

We consider several geometric properties about a  mixed super quasi-Einstein manifold with Ricci-Bourguignon soliton, then we obtain the following results.
\begin{theorem}
Let $(\mathcal{M}^{n},g)$ be a mixed super quasi-Einstein manifold  with Ricci-Bourguignon soliton $(g,\xi_{1},\lambda,\rho)$, then the integral curves of $\xi_{1}$  are geodesic if and only if the manifold $\mathcal{M}$ is a pseudo generalized quasi-Einstein manifold.
\end{theorem}
\begin{proof}
Since $(g,\xi_{1},\lambda,\rho)$ is a  Ricci-Bourguignon soliton on  $\mathcal{M}$,  then  (\ref{4.1}) becomes
\begin{equation}\label{4.3}
\begin{split}
\mathcal{L}_{\xi_{1}}g(X,Y)&+2\big(\Psi_{1}g(X,Y)+\Psi_{2}\mathcal{A}(X)\mathcal{A}(Y)+\Psi_{3} \mathcal{B}(X)\mathcal{B}(Y)\\
&+\Psi_{4}\big(\mathcal{A}(X)\mathcal{B}(Y)+\mathcal{B}(X)\mathcal{A}(Y)\big)+\Psi_{5}\mathcal{D}(X,Y)\big)\\
=&2(\lambda+\rho r)g(X,Y).
\end{split}
\end{equation}
By setting $X=\xi_{1}$ into (\ref{4.3}), we obtain
\begin{equation}\label{4.4}
\begin{split}
g(\nabla_{\xi_{1}}\xi_{1},Y)+2\big((\Psi_{1}+\Psi_{2})\mathcal{A}(Y)+\Psi_{4} \mathcal{B}(Y)\big)=2(\lambda+\rho r)\mathcal{A}(Y).
\end{split}
\end{equation}
Substituting $Y=\xi_{1}$ into (\ref{4.4}) leads to
\begin{equation}\label{4.5}
\begin{split}
\Psi_{1}+\Psi_{2}=\lambda+\rho  r.
\end{split}
\end{equation}
Consequently, equation (\ref{4.4}) simplifies to
\begin{equation}\label{4.6}
\begin{split}
g(\nabla_{\xi_{1}}\xi_{1},Y)=-2\Psi_{4} \mathcal{B}(Y).
\end{split}
\end{equation}
Supposing that the integral curves of $\xi_{1}$ are geodesic on $\mathcal{M}$, then
{\ref{4.6}} reduces to
\begin{equation}\label{4.7}
\begin{split}
-2\Psi_{4} \mathcal{B}(Y)=0.
\end{split}
\end{equation}
Putting $Y=\xi_{1}$ into (\ref{4.7}) results in $\Psi_{4}=0$. This implies that
$\mathcal{M}$ is a pseudo generalized quasi-Einstein manifold.

Conversely, assuming that $\mathcal{M}$ is a pseudo generalized quasi-Einstein manifold, then  we can get
\begin{equation*}
\begin{split}
g(\nabla_{\xi_{1}}\xi_{1},Y)=0,
\end{split}
\end{equation*}
which means that $\nabla_{\xi_{1}}\xi_{1}=0$.
\end{proof}	

From (\ref{4.5}), we can obtain $\lambda=(1-\rho n)\Psi_{1}+(1-\rho )\Psi_{2}-\rho \Psi_{3}$. Therefore, we have the following corollaries.

\begin{corollary}
Let $(\mathcal{M}^{n},g)$ be a mixed super quasi-Einstein manifold with Ricci soliton, then the soliton is expanding, steady and shrinking  according to $\Psi_{1}+\Psi_{2}>0$, $\Psi_{1}+\Psi_{2}=0$ and $\Psi_{1}+\Psi_{2}<0$.
\end{corollary}
\begin{corollary}
Let $(\mathcal{M}^{n},g)$ be a mixed super quasi-Einstein manifold  with Einstein soliton, then the soliton is expanding, steady and shrinking  according to $(n-2)\Psi_{1}-\Psi_{2}+\Psi_{3}<0$, $(n-2)\Psi_{1}-\Psi_{2}+\Psi_{3}=0$ and $(n-2)\Psi_{1}-\Psi_{2}+\Psi_{3}>0$.
\end{corollary}
\begin{corollary}
Let $(\mathcal{M}^{n},g)$ be a mixed super quasi-Einstein manifold  with traceless Ricci soliton, then the soliton is expanding, steady and shrinking  according to $(n-1)\Psi_{2}-\Psi_{3}>0$, $(n-1)\Psi_{2}-\Psi_{3}=0$ and $(n-1)\Psi_{2}-\Psi_{3}<0$.
\end{corollary}
\begin{corollary}
Let $(\mathcal{M}^{n},g)$ be a mixed super quasi-Einstein manifold  with Schouten soliton, then the soliton is expanding, steady and shrinking  according to $(n-2)\Psi_{1}+(2n-3)\Psi_{2}-\Psi_{3}>0$, $(n-2)\Psi_{1}+(2n-3)\Psi_{2}-\Psi_{3}=0$
and $(n-2)\Psi_{1}+(2n-3)\Psi_{2}-\Psi_{3}<0$.
\end{corollary}

\begin{theorem}
Let $(\mathcal{M}^{n},g)$  be a  mixed super quasi-Einstein manifold  with Ricci-Bourguignon soliton $(g,\xi_{1},\lambda,\rho)$. If the generator $\xi_{1}$  is torse-forming vector field, then the  manifold $\mathcal{M}$ is a pseudo generalized quasi-Einstein manifold and $\mathcal{D}(\xi_{2},\xi_{2})$ is an eigenvalue of  $\mathcal{D}$ corresponding to the eigenvector $\xi_{2}$.
\end{theorem}
\begin{proof}
Since $\xi_{1}$ is torse-forming vector field, then we have
\begin{equation}\label{4.8}
\begin{split}
\nabla_{X}\xi_{1}=fX+\alpha(X)\xi_{1}.
\end{split}
\end{equation}
Taking the inner product of(\ref{4.8}) with $\xi_{1}$, we can write
\begin{equation}\label{4.9}
\begin{split}
\alpha(X)=-f\mathcal{A}(X).
\end{split}
\end{equation}
In view of (\ref{4.8}) and (\ref{4.9}), then
\begin{equation}\label{4.10}
\begin{split}
\nabla_{X}\xi_{1}=f\big(X-\mathcal{A}(X)\xi_{1}\big).
\end{split}
\end{equation}
Taking $X=\xi_{1}$ in (\ref{4.10}), we can arrive at
\begin{equation}\label{4.11}
\begin{split}
\nabla_{\xi_{1}}\xi_{1}=0.
\end{split}
\end{equation}
Noting that $(g,\xi_{1},\lambda,\rho)$ is a Ricci-Bourguignon soliton on $\mathcal{M}$, thus
\begin{equation}\label{4.12}
\begin{split}
\frac{1}{2}&\big(g(\nabla_{X}\xi_{1},Y)+g(X,\nabla_{Y}\xi_{1})\big)+\Psi_{1}g(X,Y)+\Psi_{2} \mathcal{A}(X)\mathcal{A}(Y)\\
&+\Psi_{3} \mathcal{B}(X)\mathcal{B}(Y)+\Psi_{4} \big(\mathcal{A}(X)\mathcal{B}(Y)+\mathcal{B}(X)\mathcal{A}(Y)\big)+\Psi_{5} \mathcal{D}(X,Y)\\
=&(\lambda+\rho r)g(X,Y).
\end{split}
\end{equation}
Substituting $X=\xi_{1}$ in (\ref{4.12}), we can obtain
\begin{equation}\label{4.13}
\begin{split}
\frac{1}{2}g(\nabla_{\xi_{1}}\xi_{1},Y)+(\Psi_{1}+\Psi_{2})\mathcal{A}(Y)+\Psi_{4} \mathcal{B}(Y)=(\lambda+\rho r)\mathcal{A}(Y).
\end{split}
\end{equation}
Combining (\ref{4.11}) and (\ref{4.13}) yields
\begin{equation}\label{4.14}
\begin{split}
(\Psi_{1}+\Psi_{2})\mathcal{A}(Y)+\Psi_{4} \mathcal{B}(Y)=(\lambda+\rho r)\mathcal{A}(Y).
\end{split}
\end{equation}
Taking $Y=\xi_{1}$ into (\ref{4.14}), we can get
\begin{equation}\label{4.15}
\begin{split}
\Psi_{1}+\Psi_{2}=\lambda+\rho r.
\end{split}
\end{equation}
Again, putting $Y=\xi_{2}$ into (\ref{4.14} gives
\begin{equation}\label{4.16}
\begin{split}
\Psi_{4}=0,
\end{split}
\end{equation}
which implies that $\mathcal{M}$ is a pseudo generalized quasi-Einstein manifold.

Using (\ref{4.10}) and (\ref{4.16}) in (\ref{4.12}), we can write
\begin{equation}\label{4.17}
\begin{split}
(\Psi_{1}+f)g(X,Y)&+(\Psi_{2}-f) \mathcal{A}(X)\mathcal{A}(Y)+\Psi_{3} \mathcal{B}(X)\mathcal{B}(Y)+\Psi_{4} \big(\mathcal{A}(X)\mathcal{B}(Y)\\
&+\mathcal{B}(X)\mathcal{A}(Y)\big)+\Psi_{5} \mathcal{D}(X,Y)=(\lambda+\rho r)g(X,Y).
\end{split}
\end{equation}
Substituting $Y=\xi_{2}$ into (\ref{4.17}), then
\begin{equation}\label{4.18}
\begin{split}
(\Psi_{1}+f+\Psi_{3})\mathcal{B}(X)+\Psi_{5} \mathcal{D}(X,\xi_{2})=(\lambda+\rho r)\mathcal{B}(X).
\end{split}
\end{equation}
Taking $Y=\xi_{2}$ into (\ref{4.18}), we obtain
\begin{equation}\label{4.19}
\begin{split}
(\Psi_{1}+f+\Psi_{3})+\Psi_{5} \mathcal{D}(\xi_{2},\xi_{2})=(\lambda+\rho r).
\end{split}
\end{equation}
From (\ref{4.15}) and (\ref{4.19}), we have
\begin{equation}\label{4.20}
\begin{split}
f=\Psi_{2}-\Psi_{3}-\Psi_{5}\mathcal{D}(\xi_{2},\xi_{2}).
\end{split}
\end{equation}
In view of (\ref{4.15}) and (\ref{4.20}), (\ref{4.18}) becomes
\begin{equation*}
\begin{split}
\mathcal{D}(X,\xi_{2})=\mathcal{D}(\xi_{2},\xi_{2})\mathcal{B}(X)=\mathcal{D}(\xi_{2},\xi_{2})g (X,\xi_{2}),
\end{split}
\end{equation*}
which means that $\mathcal{D}(\xi_{2},\xi_{2})$ is an eigenvalue of  $\mathcal{D}$ corresponding to the eigenvector $\xi_{2}$.
\end{proof}
\begin{theorem}
Let $(\mathcal{M}^{n},g)$ be a conharmonically flat mixed super quasi-Einstein manifold with Ricci-Bourguignon soliton $(g,\xi_{1},\lambda,\rho)$. Then $n\Psi_{1}+\Psi_{2}+\Psi_{3}=0$ and the soliton is steady if and only if the generator $\xi_{1}$ is divergence-free.
\end{theorem}
\begin{proof}
Since $\mathcal{M}$ is conharmonically flat, then
\begin{equation}\label{4.21}
\begin{split}
{\rm R}(X,Y,Z,W)=&\displaystyle\frac{1}{n-2}\big({\rm Ric}(Y,Z)g(X,W)-{\rm Ric}(X,Z)g(Y,W)\\
&+g(Y,Z){\rm Ric}(X,W)-g(X,Z){\rm Ric}(Y,W)\big).
\end{split}
\end{equation}
Contracting $X$ and $W$ in (\ref{4.21}) leads to
\begin{equation}\label{4.22}
\begin{split}
r=0,
\end{split}
\end{equation}
which implies that  $n\Psi_{1}+\Psi_{2}+\Psi_{3}=0.$

Combining (\ref{4.1}) and  (\ref{4.22}), one obtains
\begin{equation}\label{4.23}
\begin{split}
\displaystyle\frac{1}{2}\mathcal{L}_{\xi_{1}}g(X,Y)+{\rm Ric}(X,Y)=\lambda g(X,Y).
\end{split}
\end{equation}
Contracting $X$ and $Y$ in (\ref{4.23}) yields
\begin{equation*}
\begin{split}
{\rm div}\xi_{1}=n\lambda.
\end{split}
\end{equation*}
This implies that ${\rm div}\xi_{1}=0\Leftrightarrow\lambda=0$.
\end{proof}

\section{Example of ${\rm M^{4}_{SQE}}$ Spacetime }

In the section, we give a non-trivaial example of  mixed super quasi-Einstein  spacetime to ensure its existence.

\begin{example} Considering a Lorentzian metric $g$  in 4-dimensional space $\mathbb{R}^{4}$ on $\mathcal{M}^{4}$ is given by
\begin{equation*}
\begin{split}
ds^{2}=g_{ij}dx^{i}dx^{j}=(dx^{1})^{2}+2{(x^{1})^{2}}(dx^{2})^{2}+2(x^{2})^{2}(dx^{3})^{2}-(dx^{4})^{2},
\end{split}
\end{equation*}
where $i,j=1,2,3,4$. The non-vanishing components of Christoffel symbols are
\begin{equation*}
\begin{split}
\Gamma_{22}^{1}=-2x^{1}, \quad \Gamma_{33}^{2}=-\frac{x^{2}}{{(x^{1})^{2}}},\quad \Gamma_{12}^{2}=\frac{1}{x^{1}}, \quad \Gamma_{23}^{3}=\frac{1}{x^{2}}.
\end{split}
\end{equation*}
Also, the nonzero curvature tensor is given by
\begin{equation*}
\begin{split}
{\rm R_{1332}}=-\frac{2x^{2}}{x^{1}}.
\end{split}
\end{equation*}
Thus the nonzero Ricci tensor is
\begin{equation*}
\begin{split}
{\rm Ric_{12}}=-\frac{1}{x^{1}x^{2}}.
\end{split}
\end{equation*}
Now, we consider the associated scalars and the symmetric tensor $\mathcal{D}$ as follows
\begin{equation*}
\begin{split}
\Psi_{1}=\frac{3}{4}e^{x^{1}},\quad \Psi_{2}=2e^{x^{1}}, \quad \Psi_{3}=-e^{x^{1}}, \quad \Psi_{4}=-x^{1}, \quad \Psi_{5}=-\frac{1}{(x^{1})^{2}}
\end{split}
\end{equation*}
and $$\mathcal{D}_{ij}=\left\{
\begin{aligned}
1,\qquad                   &if\quad i=j=1,3,\\
-2,\qquad                  &if\quad i=j=2,\\
\frac{x^{1}}{x^{2}},\qquad &if\quad i=1, j=2,\\
0,\qquad                   &otherwise.
\end{aligned}
\right.$$
Let the 1-forms are defined by
$$\mathcal{A}_{i}(x)=\left\{
\begin{aligned}
1,\qquad                  &if\quad i=4,\\
0,\qquad                  &if\quad i=1,2,3\\
\end{aligned}
\right.
\quad and \quad
\mathcal{B}_{i}(x)=\left\{
\begin{aligned}
x^{1},\qquad               &if\quad i=2,\\
x^{2},\qquad               &if\quad i=3,\\
0,\qquad                   &if\quad i=1,4.\\
\end{aligned}
\right.$$
Thus
\begin{equation}\label{5.1}
\begin{split}
{\rm Ric_{12}}=\Psi_{1}g_{12}+\Psi_{2}\mathcal{A}_{1}\mathcal{A}_{2}+\Psi_{3}\mathcal{B}_{1}\mathcal{B}_{2}+\Psi_{4}(\mathcal{A}_{1}\mathcal{B}_{2}+\mathcal{B}_{1}\mathcal{A}_{2})+\Psi_{5}\mathcal{D}_{12}.
\end{split}
\end{equation}
It can be easily seen that (\ref{5.1}) is true and the trace of the tensor $\mathcal{D}$ is zero. In particular, the 1-forms $\mathcal{A}$ and $\mathcal{B}$ follow the condition
\begin{equation*}
\begin{split}
g^{ij}\mathcal{A}_{i}\mathcal{A}_{j}=-1, \quad g^{ij}\mathcal{B}_{i}\mathcal{B}_{j}=1, \quad g^{ij}\mathcal{A}_{i}\mathcal{B}_{j}=0.
\end{split}
\end{equation*}\\
Therefore, the manifold is a mixed super quasi-Einstein spacetime.
\end{example}

\vskip 1cm

\section*{Acknowledgments}

This work was supported in part by the National Natural Science Foundation of
China (12061014), and in part by Guangxi Natural Science Foundation (2025GXNSFAA069316).

\clearpage
\end{CJK*}

\vskip 1cm

\begin{thebibliography}{30}

\bibitem{ref1}
A. Bhattacharyya, M. Tarafdar and D. Debnath. On mixed super quasi-Einstein manifolds[J]. Differential Geometry Dynamical Systems, 2008, 10(10): 44-57.

\bibitem{ref2}
M. C. Chaki. On generalized quasi-Einstein manifolds[J]. Publlicationes Math Debrecen, 2001, 58(4): 683-691.

\bibitem{ref3}
M. C. Chaki. On super quasi-Einstein manifolds[J]. Publicationes Mathematicae Debrecen, 2004, 64(3): 481-488.

\bibitem{ref4}
M. C. Chaki and M. L. Ghosh. On quasi-Einstein manifolds[J]. Indian Journal of Mathematics, 2000, 42(2): 211-220.

\bibitem{ref5}
U. De and G. C. Ghosh. On quasi-Einstein and special quasi-Einstein manifolds[C]. Proc. of the Int. Conf. of Mathematics and its applications, Kuwait University, 2004: 178-191.

\bibitem{ref6}
U. C. De. On generalized Ricci recurrent manifolds[C]. 13th Serbian Mathematical Congress, 2014.

\bibitem{ref7}
U. C. De and A. K. Gazi. On nearly quasi-Einstein manifolds[J]. Novi Sad Journal of Mathematics, 2008, 38(2): 115-121.

\bibitem{ref8}
U. C. De and G. C. Ghosh. On generalized quasi-Einstein manifolds[J]. Kyungpook Mathematical Journal, 2004, 44(4): 607-607.

\bibitem{ref9}
D. Debnath and N. Basu. On some properties of mixed super quasi-Einstein manifolds[J]. Mathematical Combinatorics, 2020: 70.

\bibitem{ref10}
S. Dey, B. Pal and A. Bhattacharyya. On some classes of mixed-super quasi-Einstein manifolds[J]. Acta Universitatis Sapientiae Mathematica, 2016, 8(1): 32-52.

\bibitem{ref11}
S. Dwivedi. Some results on Ricci-Bourguignon solitons and almost solitons[J]. Canadian Mathematical Bulletin, 2021, 64(3): 591-604.

\bibitem{ref12}
F. Fu, Y. Han and P. Zhao. Geometric and physical characteristics of mixed super quasi-Einstein manifolds[J]. International Journal of Geometric Methods in Modern Physics, 2019, 16(07): 1950104.

\bibitem{ref13}
R. S. Hamilton. Three-manifolds with positive Ricci curvature[J]. Journal of Differential geometry, 1982, 17(2): 255-306.

\bibitem{ref14}
Z. Huang, F. Su and W. Lu. The extended quasi-Einstein manifolds[J]. Filomat, 2024, 38(11): 3761-3775.

\bibitem{ref15}
Z. Huang, H. Zhou and W. Lu. Several geometric properties on mixed quasi-Einstein manifolds with soliton structures[J]. International Journal of Geometric Methods in Modern Physics, 2023, 20(10): 2350164.

\bibitem{ref16}
Y. Ishii. On conharmonic transformations[J]. Tensor, NS, 1957, 11: 73-80.

\bibitem{ref17}
M. Kon and K. Yano. Structures on manifolds[M]. World Scientific, 1985.

\bibitem{ref18}
S. Mallick and U. C. De. On mixed quasi-Einstein manifolds[J]. Ann. Univ. Sci. Budapest. E\"{o}tv\"{o}s Sect. Math, 2014, 57: 59-73.

\bibitem{ref19}
B. O'Neill. Semi-Riemannian geometry with applications to relativity[M]. Academic Press, 1983.

\bibitem{ref20}
A. Patra and A. Patra. A note on mixed super quasi-Einstein manifold[J], Malaya Journal of Matematik, 2020, 1: 530-535.

\bibitem{ref21}
D. S. Patra. Some characterizations of Einstein solitons on Sasakian manifolds[J], Canadian Mathematical Bulletin, 2022, 65(4): 1036-1049.

\bibitem{ref22}
E. Patterson. Some theorems on Ricci-recurrent spaces[J]. Journal of the London Mathematical Society, 1952, 1(3): 287-295.

\bibitem{ref23}
A. Z. Petrov. Einstein spaces[M], Elsevier, 2016.

\bibitem{ref24}
A. Sardara, C. Woob and U. C. Dec. Almost Schouten solitons and perfect fluid spacetimes[J]. Filomat, 2024, 38(16): 5827-5837.

\bibitem{ref25}
J. A. Schouten. Ricci-calculus: an introduction to tensor analysis and its geometrical applications[M]. Springer, 2013.

\bibitem{ref26}
A. Shaikh. On pseudo quasi-Einstein manifolds[J]. Periodica Mathematica Hungarica, 2009, 59(2): 119-146.

\bibitem{ref27}
A. Shaikh and S. K. Jana. On pseudo generalized quasi-Einstein manifolds[J]. Tamkang Journal of Mathematics, 2008, 39(1): 9-24.

\bibitem{ref28}
A. A. Shaikh and H. Kundu. Some curvature restricted geometric structures for projective curvature tensors[J]. International Journal of Geometric Methods in Modern Physics, 2018, 15(09): 1850157.

\bibitem{ref29}
M. Vasiulla, M. Ali and I. U\"{n}al. A study of mixed super quasi-Einstein manifolds with applications to general relativity[J]. International Journal of Geometric Methods in Modern Physics, 2024, 21(9): 2450177-535.

\bibitem{ref30}
K. Yano. On the torse-forming directions in Riemannian spaces[J]. Proceedings of the Imperial Academy, 1944, 20(6): 340-345.

\end{thebibliography}
\end{document}